\documentclass[11pt]{article}
\usepackage{amsthm}
\usepackage{amsfonts}
\usepackage{amsmath}
\usepackage{colonequals}
\usepackage{epsfig}
\usepackage{url}
\usepackage{enumitem}

\newtheorem{theorem}{Theorem}
\newtheorem{corollary}[theorem]{Corollary}
\newtheorem{lemma}[theorem]{Lemma}

\newcommand{\col}{\mbox{col}}

\title{\textbf{Planar graphs have two-coloring number at most
    $8$}\thanks{Supported by project GA14-19503S (Graph Colouring and
    Structure) of the Czech Science Foundation.}}
\author{Zden\v{e}k Dvo\v{r}\'ak\thanks{Computer Science Institute, Charles University, Prague, Czech Republic. E-mail: {\tt rakdver@iuuk.mff.cuni.cz}.}
\and Adam Kabela\thanks{Department of Mathematics, University of West Bohemia, Pilsen, Czech Republic. E-mail: {\tt kabela@ntis.zcu.cz}.}
\and Tom\'{a}\v{s} Kaiser\thanks{Department of Mathematics, Institute for Theoretical Computer Science
(CE-ITI), and European Centre of Excellence NTIS (New Technologies for
the Information Society), University of West Bohemia, Pilsen, Czech Republic. E-mail: {\tt kaisert@kma.zcu.cz}.}}
\date{}

\begin{document}
\maketitle

\begin{abstract}
  We prove that the two-coloring number of any planar graph is at
  most $8$. This resolves a question of Kierstead et al. [SIAM
  J.~Discrete Math.~23 (2009), 1548--1560]. The result is optimal.
\end{abstract}

%%%%%%%%%%%%%%%%%%%%%%%%%%%%%%%%%%%%%%%%%%%%%%%%%%%%%%%%%%%%%%%%%%%%%%

\section{Introduction}
\label{sec:introduction}

We study the two-coloring number of graphs. This parameter was
introduced by Chen and Schelp~\cite{arr2} under the name of
$p$-arrangeability; they related it to the Ramsey numbers of graphs
and the Burr--Erd\H{o}s conjecture~\cite{Burr1975}. It was
subsequently found to be related to coloring properties of graphs,
such as the game chromatic number, the acyclic chromatic number or the
degenerate chromatic number (see~\cite{Kierstead2009} and the
references therein).

We now recall the definition of the two-coloring number. Let $G$ be a
graph and let $\prec$ be a linear ordering of its vertices. (In this
paper, graphs are allowed to have parallel edges, but not loops.) For
a vertex $v\in V(G)$, let $L_{G,\prec}(v)$ be the set consisting of
the vertices $u\in V(G)$ such that $u\prec v$ and either
\begin{itemize}
\item $uv\in E(G)$, or
\item $u$ and $v$ have a common neighbor $w\in V(G)$ such that $v\prec w$.
\end{itemize}
We say that an ordering $\prec$ is \emph{$d$-two-degenerate} if
$|L_{G,\prec}(v)|\le d$ for every $v\in V(G)$.  The \emph{two-coloring
  number} $\col_2(G)$ of $G$ is defined as $d+1$ for the smallest
integer $d$ such that there exists a $d$-two-degenerate ordering of
the vertices of $G$.

Already in~\cite{arr2}, the two-coloring number of planar graphs
was bounded by an absolute constant, namely $761$. The bound was
improved to $10$ in~\cite{arr3} and eventually to $9$
in~\cite{Kierstead2009}. On the other hand, a planar graph with
two-coloring number equal to $8$ was constructed
in~\cite{arr3}. Kierstead et al.~\cite{Kierstead2009} found
simpler examples yielding the same lower bound (namely, any
$5$-connected triangulation in which the degree $5$ vertices are
non-adjacent has this property) and asked whether the two-coloring
number of all planar graphs is bounded by $8$.

We answer this question in the affirmative:
\begin{theorem}\label{thm-main}
  The two-coloring number of any planar graph is at most $8$.
\end{theorem}

It was observed in~\cite{Kierstead2009} that the list-degenerate
chromatic number of a graph $G$ is bounded by the two-coloring number
of $G$. By this observation, Theorem~\ref{thm-main} improves the known
upper bound for the list-degenerate chromatic number of planar graphs,
as well as for the ordinary degenerate chromatic number of planar
graphs, to $8$.

The structure of this paper is as follows. In the remainder of this
section, we formulate a more general version of Theorem~\ref{thm-main}
that is better suited for an inductive proof
(Theorem~\ref{thm-maingen} below). Section~\ref{sec:structural}
focuses on the basic structural properties of a hypothetical minimal
counterexample. These properties are used in
Section~\ref{sec:discharging} in a discharging procedure that provides
a contradiction, establishing Theorem~\ref{thm-maingen} and hence also
Theorem~\ref{thm-main}.

It will be useful to consider the following relative version of the
notion of $d$-two-degenerate ordering. Let $G$ be a graph, let $C$
be a subset of its vertices and let $\prec$ be a linear ordering of
$V(G)\setminus C$.  For a vertex $v\in V(G)\setminus C$, let
$L_{G,C,\prec}(v)$ be the set consisting of the vertices $u\in
V(G)\setminus C$ such that $u\prec v$ and either
\begin{itemize}
\item $uv\in E(G)$, or
\item $u$ and $v$ have a common neighbor $w\in V(G)\setminus C$ such
  that $v\prec w$, or
\item $u$ and $v$ have a common neighbor $w\in C$.
\end{itemize}
We say that an ordering $\prec$ is \emph{$d$-two-degenerate relative
  to $C$} if $|L_{G,C,\prec}(v)|\le d$ for every $v\in V(G)\setminus
C$.

\begin{theorem}\label{thm-maingen}
  Let $G$ be a plane graph and let $K$ be a set of at most three
  vertices incident with the outer face of $G$.  Let $C$ be a subset
  of $V(G)$ disjoint from $K$ such that every vertex of $C$ has at
  most $4$ neighbors in $V(G)\setminus C$.  There exists an ordering
  $\prec$ of $V(G)\setminus C$ that is $7$-two-degenerate relative
  to $C$, such that $u\prec v$ for every $u\in K$ and $v\in
  V(G)\setminus (C\cup K)$.
\end{theorem}

Note that Theorem~\ref{thm-main} follows from
Theorem~\ref{thm-maingen} by setting $C=K=\emptyset$.

%%%%%%%%%%%%%%%%%%%%%%%%%%%%%%%%%%%%%%%%%%%%%%%%%%%%%%%%%%%%%%%%%%%%%%

\section{Basic properties of a minimal counterexample}
\label{sec:structural}

Before we embark on the study of the properties of a minimal
counterexample to Theorem~\ref{thm-maingen}, let us define the notion of minimality more precisely.

A \emph{target} is a triple $(G,K,C)$, where $G$ is a plane graph, $K$ is the set of all vertices incident with the outer face of $G$, $2\le |K|\le 3$, and $C$ is a subset of $V(G)$ disjoint from $K$
such that every vertex of $C$ has at most $4$ neighbors in $V(G)\setminus C$.
Note that it suffices to show that Theorem~\ref{thm-maingen} holds for every target, since if $|K|\le 1$, then we can add $2-|K|$ new isolated vertices into the outer face of $G$ and include them in $K$,
and we can add edges between the vertices of $K$ to ensure that the outer face of $G$ is only incident with the vertices of $K$.  
An ordering $\prec$ of $V(G)\setminus C$ is \emph{valid} if $\prec$ is $7$-two-degenerate relative to $C$ and $u\prec v$ for every $u\in K$ and $v\in V(G)\setminus (C\cup K)$.
We say that a target $(G,K,C)$ is a \emph{counterexample} if there exists no valid ordering $\prec$ of $V(G)\setminus C$.
Let $s(G,K,C)=(n,-c,e_C,q,-t,e)$, where $n=|V(G)|$,  $c=|C|$, $e_C$ is the number of edges of $G$ with at least one end in $C$, 
$q$ is the number of components of $G$, $t$ is the number of triangular faces of $G$, and $e=|E(G)|$.
A target $(G',K',C')$ is \emph{smaller} than $(G,K,C)$ if $s(G',K',C')$ is lexicographically smaller than $s(G,K,C)$ (observe that this establishes a well-quasiordering on targets).
We say that a counterexample is \emph{minimal} if there exists no smaller counterexample.

In a series of lemmas, we now establish the basic properties of
minimal counterexamples.

\begin{lemma}\label{lemma-conn}
If $(G,K,C)$ is a minimal counterexample, then the following hold:
\begin{enumerate}[label=(\roman*)]
\item $C$ is an independent set,
\item all vertices of $C$ have degree $4$,
\item $G$ is connected, and
\item all faces of $G$ except possibly for the outer face have length $3$.
\end{enumerate}
\end{lemma}
\begin{proof}
  We prove claim (i). If an edge $e\in E(G)$ has both ends in $C$,
  then $(G-e,K,C)$ is a target smaller than $(G,K,C)$, and by the
  minimality of $(G,K,C)$, there exists a valid ordering $\prec$ for
  the target $(G-e,K,C)$.  Note that
  $L_{G,C,\prec}(v)=L_{G-e,C,\prec}(v)$ for every
  $v\in V(G)\setminus C$, and thus $\prec$ is also valid for the
  target $(G,K,C)$, which is a contradiction.  Hence, $C$ is an
  independent set in $G$.

  We continue with claim (iii). Suppose that $G$ is not connected.
  Hence, $G$ contains a face $f$ incident with at least two distinct
  components $G_1$ and $G_2$ of $G$.  If $G_1$ or $G_2$ consists of
  only one vertex $v\in C$, then $(G-v,K,C\setminus \{v\})$ is a
  target smaller than $(G,K,C)$ and its valid ordering is also valid
  for $(G,K,C)$, which is a contradiction.  Otherwise, since $C$ is an
  independent set, there exist vertices $v_1\in V(G_1)\setminus C$ and
  $v_2\in V(G_2)\setminus C$ incident with $f$.  Then,
  $(G+v_1v_2,K,C)$ is a target smaller than $(G,K,C)$ (with fewer
  components) and its valid ordering is also valid for $(G,K,C)$,
  which is a contradiction.  Hence, $G$ is connected.

  We next prove claim (iv). Suppose that $G$ has a non-outer face $f$
  of length other than three.  If $f$ has length $2$ and not all its
  edges belong to the boundary of the outer face, then removing one of
  its edges results in a target smaller than $(G,K,C)$ whose valid
  ordering is also valid for $(G,K,C)$, which is a contradiction.  If
  $f$ has length two and all its edges belong to the boundary of the
  outer face, then $V(G)=K$ and $(G,K,C)$ has a valid ordering, which
  is a contradiction.  Hence, $f$ has length at least $4$.  Let
  $f=v_1v_2v_3v_4\ldots$, with the labels chosen so that $v_2\in C$ if
  any vertex of $C$ is incident with $f$.  Since $C$ is an independent
  set, it follows that $v_1,v_3\not\in C$.  If $v_1\neq v_3$, then
  $(G+v_1v_3,K,C)$ is a target smaller than $(G,K,C)$ (with more
  triangular faces) and its valid ordering is also valid for
  $(G,K,C)$, which is a contradiction.  Hence, $v_1=v_3$.

  If $v_2\in C$ and $v_2$ has degree at least two, then removing one
  of at least two edges between $v_2$ and $v_1=v_3$ results in a
  target smaller than $(G,K,C)$ whose valid ordering is also valid for
  $(G,K,C)$.  If $v_2\in C$ and $v_2$ has degree exactly one, then
  $(G-v_2,K,C\setminus \{v_2\})$ is a target smaller than $(G,K,C)$
  whose valid ordering is also valid for $(G,K,C)$.  In both cases, we
  obtain a contradiction, and thus $v_2\not\in C$.

  By the choice of the labels of $f$, it follows that no vertex of $C$
  is incident with $f$.  Furthermore, note that $v_1=v_3$ is a cut in
  $G$, and thus $v_2\neq v_4$.  Consequently, $(G+v_2v_4,K,C)$ is a
  target smaller than $(G,K,C)$ whose valid ordering is also valid for
  $(G,K,C)$.  This contradiction shows that every non-outer face of
  $G$ has length three.

  It remains to prove claim (ii). Suppose that a vertex $v\in C$ has
  degree at most three.  Since $v\not\in K$, the faces incident with
  $v$ have length three, and thus the neighborhood of $v$ forms a
  clique in $G$.  The target $(G-v,K,C\setminus \{v\})$ is smaller
  than $(G,K,C)$, and thus it has a valid ordering $\prec$.  Suppose
  that for some vertices $x,y\in V(G)\setminus C$, we have
  $x\in L_{G,C,\prec}(y)$.  If $v$ is not a common neighbor of $x$ and
  $y$, then clearly $x\in L_{G-v,C\setminus \{v\},\prec}(y)$.  If $v$
  is a common neighbor of $x$ and $y$, then $x$ and $y$ are adjacent,
  and thus $x\in L_{G-v,C\setminus \{v\},\prec}(y)$.  It follows that
  $L_{G,C,\prec}(y)=L_{G-v,C\setminus \{v\},\prec}(y)$ for every
  $y\in V(G)\setminus C$, and thus $\prec$ is a valid ordering for
  $(G,K,C)$.  This is a contradiction, and thus all vertices of $C$
  have degree at least $4$.

  Note that a vertex of $C$ is not incident with parallel edges, as
  suppressing them would result in a target smaller than $(G,K,C)$
  whose valid ordering is also valid for $(G,K,C)$.  Since $C$ is an
  independent set and every vertex of $C$ has at most $4$ neighbors
  not in $C$, it follows that all vertices in $C$ have degree exactly
  $4$.
\end{proof}

Consider a target $(G,K,C)$. For a vertex $v\in V(G)\setminus C$,
let $a(v)$ be the number of neighbors of $v$ in $V(G)\setminus C$, and
let $b(v)$ be the number of neighbors of $v$ in $C$ (counted with
multiplicity when $v$ is incident with parallel edges). For
non-negative integers $a,b$, we say that $v$ is an
\emph{$(a,b)$-vertex} if $a(v)=a$ and $b(v)=b$. Similarly, we say that
$v$ is an \emph{$(a,\le\!b')$-vertex} if $a(v) = a$ and
$b(v) \leq b'$.

\begin{corollary}\label{cor-non44}
If $(G,K,C)$ is a minimal counterexample and $v\in V(G)\setminus C$,
then $a(v) \geq b(v)$. Furthermore, if $a(v)=b(v)$ and $u\in
V(G)\setminus C$ is a neighbor of $v$, then $b(u)\geq 2$.
\end{corollary}
\begin{proof}
  If $v\not\in K$, then all faces incident with $v$ are triangles.  If
  $v\in K$, then all faces except possibly for the outer one are
  triangles, and no vertex of the outer face belongs to $C$.  Since
  $C$ is an independent set, at most half of the neighbors of $v$
  belong to $C$, and thus $b(v)\le a(v)$.  Furthermore, if
  $b(v)=a(v)$, then every second neighbor of $v$ belongs to $C$, and
  thus $u$ and $v$ have two common neighbors belonging to $C$.
\end{proof}

\begin{lemma}\label{lemma-deg4}
  If $(G,K,C)$ is a minimal counterexample and
  $v\in V(G)\setminus (K\cup C)$, then $a(v) \geq 4$, and if $a(v)=4$,
  then $b(v)=4$.
\end{lemma}
\begin{proof}
  Suppose for a contradiction that $v\in V(G)\setminus (K\cup C)$
  satisfies either $a(v)\le 3$, or $a(v)=4$ and $b(v)\le 3$. By
  Corollary~\ref{cor-non44}, in the former case we have
  $b(v)\le a(v)$.

  Since $v$ has at most $4$ neighbors in $V(G)\setminus C$, it follows
  that $(G,K,C\cup \{v\})$ is a target. Note that $(G,K,C\cup \{v\})$
  is smaller than $(G,K,C)$, and let $\prec$ be its valid ordering.
  Extend $\prec$ to $V(G)\setminus C$ by letting $u\prec v$ for every
  $u\in V(G)\setminus (C\cup \{v\})$.  Note that
  $L_{G,C\cup \{v\},\prec}(w)=L_{G,C,\prec}(w)$ for every
  $w\in V(G)\setminus (C\cup \{v\})$.  Furthermore, $L_{G,C,\prec}(v)$
  contains only the neighbors of $v$ that do not belong to $C$, and
  the vertices $z$ such that $z$ and $v$ have a common neighbor
  $w\in C$.  However, since all faces of $G$ incident with $v$ are
  triangles and all vertices in $C$ have degree $4$, each neighbor
  $w\in C$ of $v$ has at most one neighbor $z$ not adjacent to $v$.
  Therefore, $|L_{G,C,\prec}(v)|\le\deg(v)=a(v)+b(v)\le 7$, and thus
  $\prec$ is a valid ordering for $(G,K,C)$.  This is a contradiction.
\end{proof}

\begin{lemma}\label{lemma-nosep}
Suppose that $(G,K,C)$ is a minimal counterexample.  If $|K|=3$, then $G$ contains no parallel edges and all triangles in $G$ bound
a face.  If $|K|=2$, then the edges bounding the outer face of $G$ are the only parallel edges in $G$, and
every non-facial triangle in $G$ contains a vertex of $C$ and both vertices of $K$.
\end{lemma}
\begin{proof}
Consider either a pair of parallel edges that do not bound the outer face of $G$, or a non-facial triangle in $G$.
Since all faces of $G$ except for the outer one have length three,
in the former case $G$ contains a non-facial cycle of length two.  Hence, let $Q$ be a non-facial cycle of length $2$ or $3$ in $G$.  

Suppose first that $V(Q)\cap C=\emptyset$.  Let $G_1$ be the subgraph of $G$ drawn in the closure of the outer face of $Q$, and let $G_2$ be the subgraph of $G$ drawn in the
closure of the inner face of $Q$.  Let $C_1=C\cap V(G_1)$ and $C_2=C\cap V(G_2)$.  Note that $(G_1, K,C_1)$ and $(G_2,V(Q),C_2)$ are targets, and since $Q$ is a non-facial cycle,
they are both smaller than $(G,K,C)$ and they have valid orderings $\prec_1$ and $\prec_2$, respectively.  Let $\prec$ be the ordering of $V(G)\setminus C$ such that
$u\prec v$ if $u,v\in V(G_1)$ and $u\prec_1 v$, or if $u,v\in V(G_2)\setminus V(Q)$ and $u\prec_2 v$, or if $u\in V(G_1)$ and $v\in V(G_2)\setminus V(Q)$.

Observe that for any $v\in V(G_1)\setminus (V(Q)\cup C_1)$, we have $L_{G,C,\prec}(v)=L_{G_1,C_1,\prec_1}(v)$, since $v$ has no neighbors in $V(G_2)$ other than those belonging to $Q$
(which are also contained in $G_1$), and since $v\prec w$ for every $w\in V(G_2)\setminus V(Q)$.  Similarly, for any $v\in V(G_2)\setminus (V(Q)\cup C_2)$, we have
$L_{G,C,\prec}(v)=L_{G_2,C_2,\prec_2}(v)$, since $v$ has no neighbors in $V(G_1)$ other than those belonging to $Q$, and all the vertices of $Q$ are contained in $G_2$
and are smaller than $v$ in both orderings $\prec$ and $\prec_2$.  Finally, for $v\in V(Q)$ we have $L_{G,C,\prec}(v)=L_{G_1,C_1,\prec_1}(v)$, since all vertices of
$V(G_2)\setminus (V(Q)\cup C_2)$ are greater than $v$ in $\prec$ and $Q$ is a clique, so all vertices of $Q$ smaller than $v$ in $\prec$ or $\prec_1$ belong to both
$L_{G,C,\prec}(v)$ and $L_{G_1,C_1,\prec_1}(v)$.  Furthermore, since $K\subseteq V(G_1)$, the choice of $\prec$ ensures that $u\prec v$ for every $u\in K$ and $v\in V(G)\setminus(C\cup K)$.
Hence, $\prec$ is a valid ordering of $(G,K,C)$, which is a contradiction.

Therefore, every non-facial $(\le\!3)$-cycle in $G$ intersects $C$.  Since $C$ is an independent set, $Q$ contains exactly one vertex of $C$.  If $Q$ has
length two, then removing one of the parallel edges of $Q$ results in a target smaller than $(G,K,C)$ whose valid ordering is also valid for $(G,K,C)$.
It follows that $G$ contains no parallel edges except possibly for those bounding its outer face, and in particular $Q$ is a triangle.

Let $Q=v_1v_2v_3$, where $v_1\in C$.  Let $e$ and $e'$ be the edges of $G$ incident with $v_1$ distinct from $v_1v_2$ and $v_1v_3$.
If exactly one of the edges $e$ and $e'$ is contained in the open disk bounded by $Q$, then
consider the neighbor $v_4$ of $v_1$ in the open disk bounded by $Q$. Since all faces incident with $v_1$ have length three, $v_4$ is adjacent to $v_2$ and $v_3$. Since the triangle $v_2v_3v_4$ does not intersect $C$, it bounds a face.  However, this implies that $v_4$
is a $(2,1)$-vertex, which contradicts Lemma~\ref{lemma-deg4}.

If neither $e$ nor $e'$ is contained in the open disk bounded by $Q$, then since $Q$ is not a facial triangle and all faces incident with $v_1$ have length three, it follows that $v_2$ and $v_3$ would be joined
by a parallel edge drawn inside the open disk bounded by $Q$; however, this is impossible, since such a parallel edge
is not incident with the outer face of $G$.
Finally, consider the case that both $e$ and $e'$ are contained in the open disk bounded by $Q$.
Similarly, $v_2$ and $v_3$ are joined by a parallel edge, and thus $K=\{v_2,v_3\}$.
We conclude that every non-facial triangle in $G$ contains a vertex of $C$ and two vertices of $K$.
\end{proof}

\begin{corollary}\label{cor-outdeg4}
If $(G,K,C)$ is a minimal counterexample, then every vertex of $K$ has degree at least $4$.
\end{corollary}
\begin{proof}
Suppose first that a vertex $v\in K$ has degree two.  Since all faces of $G$ except for the outer one are triangles,
if $|K|=2$, this would imply that $G$ contains a loop, which is a contradiction.  If $|K|=3$, then since all faces of $G$
are triangles and $G$ does not contain parallel edges, we have $V(G)=K$, and any ordering of $V(G)$ is valid,
which is a contradiction.

Next, suppose that $v$ has degree three, and let $x$ be the neighbor of $v$ not belonging to $K$.
If $|K|=2$, then since all faces incident with $x$ are triangles and $x$ is not incident with a parallel edge,
it follows that $V(G)=K\cup \{x\}$ and $x$ has degree two.
If $|K|=3$, say $K=\{v,y_1,y_2\}$, then since all faces of $G$ are triangles, it follows that
$vxy_1$ and $vxy_2$ are triangles.  Also, every triangle in $G$ is facial, and thus $x$ has degree three.
In both cases, $x\not\in C$ and $x$ is a $(2,0)$-vertex or a $(3,0)$-vertex, which contradicts Lemma~\ref{lemma-deg4}.
\end{proof}

Let $\prec$ be an ordering of $V(G)\setminus C$ in a target $(G,K,C)$.  For adjacent vertices $u\in V(G)\setminus C$ and $v$, a vertex $w\in V(G)\setminus C$ distinct from $u$ is a
\emph{friend of $u$ via $v$} if $w\prec u$ and
\begin{itemize}
\item $w=v$, or
\item $vw\in E(G)$, $uw\not\in E(G)$, and $v\in C$, or
\item $vw\in E(G)$, $uw\not\in E(G)$, $u$ and $w$ do not have a common neighbor in $C$, and $u\prec v$.
\end{itemize}
Note that $L_{G,C,\prec}(u)$ consists exactly of the friends of $u$ via its neighbors.  We will frequently use the following observations.
\begin{lemma}\label{lemma-omfr}
Let $(G,K,C)$ be a minimal counterexample and let $u\in V(G)\setminus
(C\cup K)$ and $v\in V(G)$ be neighbors.  Let $\prec$ be an ordering
of $V(G)\setminus C$. Then the following hold:
\begin{enumerate}[label=(\roman*)]
\item if $v\prec u$ or $v\in C$, then $u$ has at most one friend via $v$,
\item if $v\not\in C\cup K$ and $u\prec v$, then $u$ has at most $a(v)-3$
  friends via $v$,
\item if $v\not\in C\cup K$, $u\prec v$, and $v$ has a neighbor
  $r\not\in C$ non-adjacent to $u$ such that $u\prec r$ and no vertex
  of $C$ is a common neighbor of $u,v$ and $r$, then $u$ has at most
  $a(v)-4$ friends via $v$.
\end{enumerate}
\end{lemma}
\begin{proof}
  (i) If $v\prec u$, then $v$ is the only friend of $u$ via $v$.  If
  $v\in C$, then since all faces incident with $v$ are triangles and
  $v$ has degree $4$, the vertex $v$ has at most one neighbor not
  adjacent to $u$, and thus $u$ has at most one friend via $v$.

  (ii) Suppose that $v\not\in C\cup K$ and $u\prec v$. By
  Lemma~\ref{lemma-deg4}, we have $a(v)\ge 4$, and since all faces
  incident with $v$ have length three, it follows that $v$ has at
  least two neighbors $z_1,z_2\not\in C$ distinct from $u$ such that
  for $i\in \{1,2\}$, either $uvz_i$ is a face, or $z_i$ and $u$ have
  a common neighbor $z'_i\in C$ such that $uvz'_i$ is a face. In
  either case, $z_i$ is not a friend of $u$ via $v$. Since $u$ is not
  a friend of $u$ via $v$, it follows that $u$ has at most $a(v)-3$
  friends via $v$.

  (iii) Let us now additionally assume that $v$ has a neighbor $r$ as
  described in the last case of the lemma. Using the notation from the
  previous case, we first show that the vertex $z_1$ is distinct from
  $r$. This is clearly the case if $z_1$ is adjacent to $u$. Suppose
  then that $z_1$ is not adjacent to $u$, and thus $z_1$ is a neighbor
  of a vertex $z'_1\in C$ such that $uvz'_1$ is a face. But then $u$,
  $v$ and $z_1$ have a common neighbor in $C$, and thus $r\neq
  z_1$. By a symmetric argument, $r\neq z_2$. 

  Since $u\prec r$, the vertex $r$ is not a friend of $u$, and thus
  $u$ has at most $a(v)-4$ friends via $v$.
\end{proof}

\begin{lemma}\label{lemma-redupath}
If $(G,K,C)$ is a minimal counterexample, then $G$ contains no path $P=v_1v_2\ldots v_k$ with $k\ge 2$ disjoint from $K$, such that $v_1$ is a $(5,\le\!1)$-vertex, $v_2,\ldots, v_{k-1}$ are $(6,0)$-vertices,
and $v_k$ is a $(5,\le\!2)$-vertex.
\end{lemma}
\begin{proof}
Suppose for a contradiction that $G$ contains such a path $P$.  Without loss of generality, $P$ is an induced path.  
Furthermore, $P$ is disjoint from $C$ by Lemma~\ref{lemma-conn}.
Note that each vertex of $P$ has at most $4$ neighbors in $V(G)\setminus (C\cup V(P))$,
and thus $(G,K,C\cup V(P))$ is a target smaller than $(G,K,C)$.  Let $\prec$ be a valid ordering of $(G,K,C\cup V(P))$, and let us extend the ordering to $(G,K,C)$ by setting
$u\prec v_1\prec v_2\prec\ldots\prec v_k$ for every $u\in V(G)\setminus (C\cup V(P))$.  Observe that $L_{G,C\cup V(P),\prec}(u)=L_{G,C,\prec}(u)$ for every $u\in V(G)\setminus (C\cup V(P))$.
By Lemma~\ref{lemma-omfr}, $v_k$ has at most one friend via each of its neighbors, and thus $|L_{G,C,\prec}(v_k)|\le 7$.  The vertex $v_{k-1}$ has at most $2$ friends via $v_k$ and at most one friend via
each of its neighbors distinct from $v_k$, and thus $|L_{G,C,\prec}(v_{k-1})|\le 7$.  Consider any $i=1,\ldots, k-2$.  By Lemma~\ref{lemma-omfr}, the vertex $v_i$ has at most $2$ friends via $v_{i+1}$
(because $v_{i+1}$ is a $(6,0)$-vertex and we can set $r=v_{i+2}$) and at most one friend via each of its neighbors distinct from $v_{i+1}$, and thus $|L_{G,C,\prec}(v_i)|\le 7$.
Therefore, $\prec$ is a valid ordering for $(G,K,C)$, which is a contradiction.
\end{proof}

\begin{lemma}\label{lemma-reducyc}
If $(G,K,C)$ is a minimal counterexample, then $G$ contains no induced cycle $Q=v_1v_2\ldots v_k$ with $k\ge 4$ disjoint from $K$, such that 
$v_k$ is a $(5,\le\!2)$-vertex and $v_1,\ldots, v_{k-1}$ are $(6,0)$-vertices.
\end{lemma}
\begin{proof}
Suppose for a contradiction that $G$ contains such an induced cycle $Q$.
Clearly, $Q$ is disjoint from $C$ by Lemma~\ref{lemma-conn}.
Note that each vertex of $Q$ has at most $4$ neighbors in $V(G)\setminus (C\cup V(Q))$,
and thus $(G,K,C\cup V(Q))$ is a target smaller than $(G,K,C)$.  Let $\prec$ be a valid ordering of $(G,K,C\cup V(Q))$, and let us extend the ordering to $(G,K,C)$ by setting
$u\prec v_1\prec v_2\prec\ldots\prec v_k$ for every $u\in V(G)\setminus (C\cup V(Q))$.  Observe that $L_{G,C\cup V(Q),\prec}(u)=L_{G,C,\prec}(u)$ for every $u\in V(G)\setminus (C\cup V(Q))$.
By Lemma~\ref{lemma-omfr}, $v_k$ has at most one friend via each of its neighbors, and thus $|L_{G,C,\prec}(v_k)|\le 7$.  The vertex $v_{k-1}$ has at most $2$ friends via $v_k$ and at most one friend via
each of its neighbors distinct from $v_k$, and thus $|L_{G,C,\prec}(v_{k-1})|\le 7$.  Consider any $i=2,\ldots, k-2$.  By Lemma~\ref{lemma-omfr}, the vertex $v_i$ has at most $2$ friends via $v_{i+1}$
(because $v_{i+1}$ is a $(6,0)$-vertex and we can set $r=v_{i+2}$) and at most one friend via each of its neighbors distinct from $v_{i+1}$, and thus $|L_{G,C,\prec}(v_i)|\le 7$.
Finally, the $(6,0)$-vertex $v_1$ has at most two friends via $v_2$, at most one friend via $v_k$ (since we can set $r=v_{k-1}$), and at most one friend via each of its neighbors distinct from
$v_2$ and $v_k$, and thus $|L_{G,C,\prec}(v_1)|\le 7$.  Therefore, $\prec$ is a valid ordering for $(G,K,C)$, which is a contradiction.
\end{proof}

\begin{lemma}\label{lemma-redu61path}
If $(G,K,C)$ is a minimal counterexample, then $G$ contains no path $P=v_1v_2\ldots v_k$ with $k\ge 3$ disjoint from $K$,
such that $v_1$ is a $(5,\le\!1)$-vertex, $v_2,\ldots, v_{k-2}$ are $(6,0)$-vertices (if $k\ge 4$), $v_{k-1}$ is a $(6,1)$-vertex and $v_k$ is a $(5,0)$-vertex.
\end{lemma}
\begin{proof}
Suppose for a contradiction that $G$ contains such a path $P$.  Without loss of generality, $P$ is an induced path ($v_k$ has no neighbors in $P$ distinct from $v_{k-1}$ by Lemma~\ref{lemma-redupath}).
Note that $P$ is disjoint from $C$, and each vertex of $P$ has at most $4$ neighbors in $V(G)\setminus (C\cup V(P))$,
and thus $(G,K,C\cup V(P))$ is a target smaller than $(G,K,C)$.  Let $\prec$ be a valid ordering of $(G,K,C\cup V(P))$, and let us extend the ordering to $(G,K,C)$ by setting
$u\prec v_1\prec v_2\prec\ldots\prec v_{k-2}\prec v_k\prec v_{k-1}$ for every $u\in V(G)\setminus (C\cup V(P))$.  Observe that $L_{G,C\cup V(P),\prec}(u)=L_{G,C,\prec}(u)$ for every $u\in V(G)\setminus (C\cup V(P))$.
By Lemma~\ref{lemma-omfr}, $v_{k-1}$ has at most one friend via each of its neighbors, and thus $|L_{G,C,\prec}(v_{k-1})|\le 7$.  The vertex $v_k$ has at most $3$ friends via $v_{k-1}$ and at most one friend via
each of its neighbors distinct from $v_{k-1}$, and thus $|L_{G,C,\prec}(v_k)|\le 7$.  Consider any $i=1,\ldots, k-2$.  By Lemma~\ref{lemma-omfr}, the vertex $v_i$ has at most $2$ friends via $v_{i+1}$
(because we can set $r=v_{i+2}$ and either $v_{i+1}$ is a $(6,0)$-vertex, or $r$ is a $(5,0)$-vertex) and at most one friend via each of its neighbors distinct from $v_{i+1}$, and thus $|L_{G,C,\prec}(v_i)|\le 7$.
Therefore, $\prec$ is a valid ordering for $(G,K,C)$, which is a contradiction.
\end{proof}

%%%%%%%%%%%%%%%%%%%%%%%%%%%%%%%%%%%%%%%%%%%%%%%%%%%%%%%%%%%%%%%%%%%%%%

\section{Discharging}
\label{sec:discharging}

Let us now proceed with the discharging phase of the proof.  Let $(G,K,C)$ be a minimal counterexample.  Let us assign charge $c'_0(v)=10\deg (v)-60$ to each vertex $v\in V(G)$.  Since all faces of
$G$ except possibly for the outer one have length three, we have
$|E(G)|=3|V(G)|-3-|K|$, and thus
\begin{align*}
  \sum_{v\in V(G)} c'_0(v)&=-60|V(G)|+10\sum_{v\in V(G)} \deg(v)\\&=-60|V(G)|+20|E(G)|=-60-20|K|.  
\end{align*}
Next, every vertex of $v\in V(G)\setminus C$ sends charge of $5$ to every adjacent vertex in $C$,
and let $c_0$ denote the resulting assignment of charge.  Since the total amount of charge does not change, we have $\sum_{v\in V(G)} c_0(v)=-60-20|K|$.  If $v\in C$, then $\deg(v)=4$, $c'_0(v)=-20$,
and $v$ receives $5$ from each of its neighbors, and thus $c_0(v)=0$.  An $(a,b)$-vertex $v\in V(G)\setminus C$ has $c'_0(v)=10a+10b-60$ and $v$ sends $5$ to $b$ of its neighbors, and thus $c_0(v)=10a+5b-60$.

We say that a vertex $v\in V(G)\setminus C$ is \emph{big} if $v\in K$ or $c_0(v)>0$ (i.e., $v$ is not a $(4,4)$-vertex, a $(5,\le\!2)$-vertex, or a $(6,0)$-vertex).
We call vertices not belonging to $K$ \emph{internal}.
Next, we redistribute the charge according to the following rules, obtaining the \emph{final charge} $c$.
\begin{description}
\item[R1] Every big vertex sends $2$ to each neighboring internal $(5,0)$-vertex.
\item[R2] Every big vertex sends $1$ to each neighboring internal $(5,1)$-vertex.
\item[R3] If $v_1v_2\ldots v_k$ with $k\ge 3$ is a path in $G$ such that $v_1xv_2$, $v_2xv_3$, \ldots, $v_{k-1}xv_k$ are faces for some vertex $x$,
$v_1$ is big, $x$ is either big or an internal $(6,0)$-vertex, $v_2$, \ldots, $v_{k-1}$ are internal $(6,0)$-vertices, and $v_k$ is an internal $(5,\le\!1)$-vertex, then $v_1$ sends $1$ to $v_k$.
\end{description}
In the case of rule R3, we say that the charge \emph{arrives to $v_k$ through pair $(v_{k-1},x)$}, and \emph{departs $v_1$ through pair $(v_2,x)$}.
Note that it is possible for charge to arrive through $(x,v_{k-1})$ or depart through $(x,v_2)$ as well, if $x$ is an internal $(6,0)$-vertex.
If the charge departs through both $(v_2,x)$ and $(x,v_2)$, we say that the edge $v_2x$ is \emph{heavy for $v_1$}.
The key observations concerning the rule R3 are the following.

\begin{lemma}\label{lemma-arrive}
Let $(G,K,C)$ be a minimal counterexample, let $v$ be an internal $(5,\le\!1)$-vertex, and let $vu_1x$ be a face of $G$.  If $u_1$ is an internal $(6,0)$-vertex, then
charge arrives to $v$ through $(u_1,x)$.
\end{lemma}
\begin{proof}
By Lemma~\ref{lemma-redupath}, $x$ is not an internal $(5,\le\!2)$-vertex, and by Corollary~\ref{cor-non44}, $x$ is not a $(4,4)$-vertex.
Hence, $x$ is either big or an internal $(6,0)$-vertex.

Let $vu_1x$, $u_1u_2x$, $u_2u_3x$, \ldots, $u_{k-1}u_kx$ be faces of $G$ incident with $x$ in order, where $k\ge 2$ is chosen minimum such that $u_k$ is not an internal $(6,0)$-vertex (possibly $u_k=v$).
If $u_k$ is big, then it sends charge to $v$ by R3 and this charge arrives through $(u_1,x)$.  Hence, assume that $u_k$ is not big.  Since $u_{k-1}$ is a $(6,0)$-vertex, Corollary~\ref{cor-non44}
implies that $u_k$ is not a $(4,4)$-vertex.  Therefore, $u_k$ is an internal $(5,\le\!2)$-vertex.  By Lemma~\ref{lemma-redupath}, it follows that $u_k=v$.
Since $x$ does not have a big neighbor, $x$ is an internal vertex. Since $x$ is internal big or $(6,0)$-vertex,
its degree is at least $6$, and thus $k\ge 6$.  However, Lemma~\ref{lemma-nosep} implies that $vu_1u_2\ldots u_{k-1}$ is an induced cycle, which contradicts Lemma~\ref{lemma-reducyc}.
\end{proof}

\begin{lemma}\label{lemma-depart}
  Let $(G,K,C)$ be a minimal counterexample, let $v$ be a big vertex,
  and let $vu_1u_2$, $vu_2u_3$, and $vu_3u_4$ be distinct faces of
  $G$.
\begin{itemize}
\item If $u_1u_2$ is heavy for $v$, and $u_1u_2w$ is the face of $G$ with $w\neq v$, then
$w$ is an internal $(5,\le\!1)$-vertex.  Furthermore, no charge departs $v$ through $(u_2,u_3)$, and $u_3u_4$ is not heavy for $v$.
\item If $u_1$ is an internal $(5,\le\!1)$-vertex, then charge does not depart $v$ through $(u_2,u_3)$.
\item If $v$ is an internal $(6,1)$-vertex adjacent to an internal $(5,0)$-vertex and charge departs $v$ through
$(u_1,u_2)$, then $u_3$ is an internal $(5,0)$-vertex.
\end{itemize}
\end{lemma}
\begin{proof}
Suppose that charge departs $v$ through both $(u_1,u_2)$ and $(u_2, u_1)$.  By the assumptions of the rule R3, both $u_1$ and $u_2$ are internal $(6,0)$-vertices.
For $i=1,2$, there exists a path starting in $u_i$, passing through internal $(6,0)$-vertices adjacent to $u_{3-i}$, and ending in an internal $(5,\le\!1)$-vertex $x_i$
adjacent to $u_{3-i}$.  By Lemma~\ref{lemma-redupath}, we have $x_1=x_2$.  Hence, $u_1u_2x_1$ is a triangle, and by Lemma~\ref{lemma-nosep}, we have $w=x_1=x_2$.
\begin{itemize}
\item Suppose that in this situation, charge departs through $(u_2,u_3)$ because of a path in the neighborhood of $u_3$
ending in an internal $(5,\le\!1)$-vertex $x$.  By Lemma~\ref{lemma-redupath}, we have $x=w$, and by Lemma~\ref{lemma-nosep}, $u_2u_3w$ bounds a face.
However, then $u_2$ has degree $4$, which is a contradiction since $u_2$ is a $(6,0)$-vertex.
\item Suppose that in this situation, $u_3u_4$ is heavy for $v$.  Then the vertex $w'\neq v$ of the face $u_3u_4w'$ is an internal $(5,\le\!1)$-vertex,
and by Lemma~\ref{lemma-redupath}, we have $w=w'$.  By Lemma~\ref{lemma-nosep}, it follows that $u_2$ and $u_3$ have degree $4$, which is a contradiction, since they are $(6,0)$-vertices.
\end{itemize}

Suppose now that $u_1$ is an internal $(5,\le\!1)$-vertex, and that charge departs $v$ through $(u_2,u_3)$ because of a path in the neighborhood of $u_3$
ending in an internal $(5,\le\!1)$-vertex $x$.  By Lemma~\ref{lemma-redupath}, we have $x=u_1$.  But then $u_3$ is adjacent to $x$, and Lemma~\ref{lemma-nosep}
would imply that $u_1u_2u_3$ is a face and $u_2$ has degree three, which is a contradiction.

Suppose that $v$ is an internal $(6,1)$-vertex adjacent to an internal $(5,0)$-vertex $z$ and that
charge departs $v$ through $(u_1,u_2)$ because of a path in the neighborhood of $u_2$ ending in an internal $(5,\le\!1)$-vertex $x$.  By Lemma~\ref{lemma-redu61path},
we have $x=z$.  But then $u_2$ is adjacent to $z$, and the triangle $u_2vz$ bounds a face by Lemma~\ref{lemma-nosep}.  Hence, $z=u_3$.
\end{proof}

Let us now analyze the final charge of the vertices of $G$.

\begin{lemma}\label{lemma-50}
Let $(G,K,C)$ be a minimal counterexample.  If $v$ is an internal $(5,0)$-vertex of $G$, then $c(v)\ge 0$.
\end{lemma}
\begin{proof}
We have $c_0(v)=-10$.

By Corollary~\ref{cor-non44} and Lemma~\ref{lemma-redupath}, every neighbor of $v$ in $G$ is either big or an internal $(6,0)$-vertex.
Suppose that $v$ is adjacent to $\beta$ big vertices; each of them sends $2$ to $v$ by the rule R1.
By Lemma~\ref{lemma-arrive}, charge arrives to $v$ through $10-2\beta$ pairs.  Hence,
$c(v)=c_0(v)+2\beta+(10-2\beta)=0$.
\end{proof}

\begin{lemma}\label{lemma-51}
Let $(G,K,C)$ be a minimal counterexample.  If $v$ is an internal $(5,1)$-vertex of $G$, then $c(v)\ge 0$.
\end{lemma}
\begin{proof}
We have $c_0(v)=-5$.

By Corollary~\ref{cor-non44} and Lemma~\ref{lemma-redupath}, all
neighbors of $v$ except for the one belonging to $C$ are either big or
internal $(6,0)$-vertices.  Let $v_1$, \ldots, $v_6$ be the neighbors
of $v$ in order, where $v_2\in C$.  Since $(6,0)$-vertices have no
neighbor in $C$, both $v_1$ and $v_3$ are big.  Let $\beta\ge 2$ be
the number of big vertices adjacent to $v$; each of them sends $1$
to $v$ by the rule R2.  By Lemma~\ref{lemma-arrive}, charge arrives to
$v$ through $10-2\beta$ pairs.  Since $\beta\le 5$,
$c(v)=c_0(v)+\beta+(10-2\beta)\ge 0$.
\end{proof}

\begin{lemma}\label{lemma-verybig}
Let $(G,K,C)$ be a minimal counterexample.  If $v$ is a big $(a,b)$-vertex, then
$c(v)\ge 8a+7b-60$.
In particular, if $v$ is internal and $v$ is neither a $(6,1)$-vertex nor a $(7,0)$-vertex, then $c(v)\ge 0$.
\end{lemma}
\begin{proof}
By Lemma~\ref{lemma-nosep}, the neighborhood of $v$ in $G$ induces a cycle, which we denote by $Q$.
If $v$ is an internal vertex or $|K|=3$, then the length of $Q$ is $a+b$.
If $v\in K$ and $|K|=2$ then the length of $Q$ is $a+b-1$. Note that if $v\in K$, then $a+b\ge 4$ by Corollary~\ref{cor-outdeg4}.

Let us define a weight $w(e)$ for an edge $e=xy$ of $Q$ as follows.  If charge departs $v$ through at least one of $(x,y)$ and $(y,x)$, then let $w(e)=2$.
If $x$ or $y$ is an internal $(5,\le\!1)$-vertex and neither $x$ nor $y$ belongs to $C$, then let $w(e)=1$.  Otherwise, let $w(e)=0$.
Note that no two internal $(5,\le\!1)$-vertices are adjacent by Lemma~\ref{lemma-redupath}, and that if charge departs $v$ through at least one of $(x,y)$ and $(y,x)$, then neither $x$ nor $y$
is an internal $(5,\le\!1)$-vertex.  Furthermore, if $xyz$ is subpath of $Q$ and $y$ is an internal $(5,0)$-vertex, then $w(xy)=w(yz)=1$.  We conclude that
$\sum_{e\in E(Q)} w(e)$ is an upper bound on the amount of charge sent by $v$.

Note that $w(e)\le 2$ for every $e\in E(Q)$, and $w(e)=0$ if $e$ is incident with a vertex of $C$.
Since $C$ is an independent set, exactly $2b$ edges of $Q$ are incident with a vertex of $C$, and thus $\sum_{e\in E(Q)} w(e)\le 2(a+b-2b)=2(a-b)$.
Therefore, $c(v)\ge c_0(v)-2(a-b)=(10a+5b-60)-2(a-b)=8a+7b-60$.

If $a\ge 8$, then $c(v)\ge 8a-60\ge 4$.  If $a=7$ and $b\ge 1$, then $c(v)\ge 8\cdot 7+7-60=3$.  If $a=6$ and $b\ge 2$, then $c(v)\ge8\cdot 6+7\cdot 2-60=2$. Finally, if $a=5$ and $b\ge 3$, then $c(v)\ge8\cdot 5+7\cdot 3-60=1$.  Hence, if $v$ is an internal big vertex, it follows that $c(v)\ge 0$
unless $a=7$ and $b=0$, or $a=6$ and $b=1$.
\end{proof}

\begin{lemma}\label{lemma-70}
Let $(G,K,C)$ be a minimal counterexample.  If $v$ is an internal $(7,0)$-vertex, then $c(v)\ge 0$.
\end{lemma}
\begin{proof}
Note that $c_0(v)=10$.
Let $v_1v_2\ldots v_7$ denote the cycle induced by the neighbors of $v$, and
let $n_5$ denote the number of internal $(5,\le\!1)$-vertices of $G$ adjacent to $v$.

Since no two internal $(5,\le\!1)$-vertices are adjacent, it follows that $n_5\le 3$.
By rules R1 and R2, the vertex $v$ sends at most $2n_5$ units of charge.  Furthermore, $v$ sends charge over at most $7-2n_5$ edges by rule R3.  If $n_5\ge 2$, then $c(v)\ge c_0(v)-2n_5-2(7-2n_5)=2(n_5-2)\ge 0$. 

If $n_5=1$, then suppose that $v_1$ is the internal $(5,\le\!1)$-vertex.  By Lemma~\ref{lemma-depart}, no charge departs $v$ through $(v_2,v_3)$ or $(v_7,v_6)$.
Also, at most one of the edges $v_3v_4$, $v_4v_5$, and $v_5v_6$ is heavy for $v$.  Therefore, charge departs $v$ through at most $6$ pairs, and
$c(v)\ge c_0(v)-2n_5-6>0$.

Finally, suppose that $n_5=0$.  By Lemma~\ref{lemma-depart}, for $i=1,\ldots, 7$, at most one of the edges $v_iv_{i+1}$, $v_{i+1}v_{i+2}$, $v_{i+2}v_{i+3}$ (with indices taken cyclically)
is heavy.  Therefore, charge departs $v$ through at most $9$ pairs.  Hence, $c(v)\ge c_0(v)-9>0$.
\end{proof}

\begin{lemma}\label{lemma-61}
Let $(G,K,C)$ be a minimal counterexample.  If $v$ is an internal $(6,1)$-vertex, then $c(v)\ge 0$.
\end{lemma}
\begin{proof}
Note that $c_0(v)=5$.
Let $Q=v_1v_2\ldots v_7$ denote the cycle induced by the neighbors of $v$, where $v_2\in C$.  

Suppose first that $v$ is adjacent to an internal $(5,0)$-vertex, to which $v$ sends $2$ by the rule R1.  By Lemma~\ref{lemma-redu61path}, $v$ is adjacent only to one internal $(5,0)$-vertex
and no other internal $(5,\le\!1)$-vertex.  Furthermore, by the third part of Lemma~\ref{lemma-depart}, charge departs $v$ through at most two pairs.
Hence, $c(v)\ge c_0(v)-2-2>0$.

Hence, we can assume that $v$ is not adjacent to internal
$(5,0)$-vertices.  Let $n_5$ be the number of internal
$(5,1)$-vertices adjacent to $v$; $v$ sends $1$ to each of them by the
rule R2.  Note that $n_5\le 3$, since no two internal $(5,1)$-vertices
are adjacent by Lemma~\ref{lemma-redupath}.  Since $v_2\in C$, neither
$v_1$ nor $v_3$ is a $(6,0)$-vertex, and thus the edges $v_1v_7$ and
$v_3v_4$ are not heavy for $v$.

Suppose first that $n_5=0$.  If no edge of $Q$ is heavy for $v$, then charge departs $v$ through at most $5$ pairs
and $c(v)\ge c_0(v)-5=0$.  Hence, by symmetry we can assume that $v_4v_5$ or $v_5v_6$ is heavy for $v$.  
Lemma~\ref{lemma-depart} implies that no other edge of $Q$ is heavy for $v$.  Let us distinguish the cases.
\begin{itemize}
\item If $v_4v_5$ is heavy, then Lemma~\ref{lemma-depart} implies that charge does not depart through the pair $(v_4,v_3)$, and it does not depart through the pair $(v_3,v_4)$ since $v_3$ is not a $(6,0)$-vertex.
\item If $v_5v_6$ is heavy, then Lemma~\ref{lemma-depart} implies that the common neighbor $w\neq v$ of $v_5$ and $v_6$ is an internal $(5,\le\!1)$-vertex,
and furthermore, that charge may only depart $v$ through pairs $(v_4,v_5)$, $(v_7,v_6)$, $(v_4,v_3)$, and $(v_7,v_1)$ in addition to $(v_5,v_6)$ and $(v_6,v_5)$.

Suppose that the charge departs $v$ through all these pairs.  By Lemma~\ref{lemma-redupath}, all the charge arrives to $w$.  However, then $w$ is adjacent to
$v_1$, $v_3$, $v_5$, $v_6$, as well as at least two $(6,0)$-vertices of the paths showing that the charge departing through the pairs $(v_4,v_5)$ and $(v_7,v_6)$ arrives to $w$.
This is a contradiction, since $w$ has at most $5$ neighbors not belonging to $C$.
\end{itemize}
In both cases, we conclude that charge departs $v$ through at most $5$ pairs, and thus $c(v)\ge c_0(v)-5=0.$

Suppose now that $n_5=1$.  If neither $v_1$ nor $v_3$ is an internal $(5,1)$-vertex, then $v$ sends charge over at most three edges by the rule R3 and at most one of them is heavy for $v$ by Lemma~\ref{lemma-depart},
and $c(v)\ge c_0(v)-n_5-4=0$.  Hence, by symmetry, we can assume that $v_3$ is an internal $(5,1)$-vertex.  By Lemma~\ref{lemma-depart}, only one of the edges $v_5v_6$ and $v_6v_7$
may be heavy.  If $v_6v_7$ is heavy, then charge does not depart $v$ through $(v_7,v_1)$ or $(v_1,v_7)$, by Lemma~\ref{lemma-depart} and since $v_1$ is not a $(6,0)$-vertex.
If $v_5v_6$ is heavy, then charge does not depart $v$ through $(v_4,v_5)$ or $(v_5,v_4)$ by Lemma~\ref{lemma-depart}.  In either case, charge departs $v$ through at most $4$ pairs, and again $c(v)\ge 0$.

Suppose that $n_5=2$.
Recall that no two internal $(5,\le\!1)$-vertices are adjacent by Lemma~\ref{lemma-redupath}.
If at least one of $v_1$ and $v_3$ is not an internal $(5,1)$-vertex, then $v$ sends charge over at most two edges by rule R3 and neither of them is heavy for $v$ by
Lemma~\ref{lemma-depart}, hence $c(v)\ge c_0(v)-n_5-2>0$.  If both $v_1$ and $v_3$ are internal $(5,1)$-vertices, then only the edge $v_5v_6$ may be heavy for $v$
by Lemma~\ref{lemma-depart}, and if it is heavy, then no charge departs $v$ through $(v_4,v_5)$, $(v_5,v_4)$, $(v_6,v_7)$ and $(v_7,v_6)$.  Hence, charge departs $v$ through at most $3$ pairs
and $c(v)\ge c_0(v)-n_5-3=0$.

Finally, suppose that $n_5=3$.  In this case, Lemma~\ref{lemma-depart} shows that no charge departs $v$, and thus $c(v)=c_0(v)-n_5>0$.
\end{proof}

\begin{proof}[Proof of Theorem~\ref{thm-maingen}]
Suppose for a contradiction that Theorem~\ref{thm-maingen} is false.  Then, there exists a minimal counterexample $(G,K,C)$.
Assign and redistribute charge among its vertices as we described above.  Note that the redistribution of the charge
does not change its total amount, and thus
$$\sum_{v\in V(G)} c(v)=\sum_{v\in V(G)} c_0(v)=-60-20|K|.$$
Recall that $c(v)=c_0(v)=0$ for every $v\in C$.  If $v$ is an internal big vertex, then $c(v)\ge 0$ by Lemmas~\ref{lemma-verybig}, \ref{lemma-70}
and \ref{lemma-61}.  If $v$ is an internal vertex with $c_0(v)$ negative, then by Lemma~\ref{lemma-deg4}, it follows that
$v$ is either a $(5,0)$-vertex, or a $(5,1)$-vertex, and $c(v)\ge 0$ by Lemmas~\ref{lemma-50} and \ref{lemma-51}.
If $v$ is an internal vertex with $c_0(v)=0$ (i.e., $v$ is a $(4,4)$-vertex, or a $(5,2)$-vertex, or a $(6,0)$-vertex),
then $c(v)=c_0(v)=0$.  Therefore,
$$\sum_{v\in V(G)} c(v)\ge \sum_{v\in K} c(v).$$
Consider an $(a,b)$-vertex $v\in K$.  Since $v$ is incident with two edges of the outer face of $G$, we have $a\ge 2$,
and $a+b\ge 4$ by Corollary~\ref{cor-outdeg4}.  By Lemma~\ref{lemma-verybig},
$c(v)\ge 8\cdot 2+7\cdot 2-60=-30$.  Therefore,
$$\sum_{v\in K} c(v)\ge -30|K|.$$
However, since $|K|\le 3$, we have $-30|K|>-60-20|K|$, which is a contradiction.  Therefore, no counterexample to Theorem~\ref{thm-maingen}
exists.
\end{proof}

%\bibliographystyle{siam}
%\bibliography{twocol}

\end{document}